\documentclass{icmart}

\contact[souravc@stanford.edu]{Department of Statistics and Department of Mathematics, Stanford University, USA}

\newtheorem{theorem}{Theorem}[section]
\newtheorem{corollary}[theorem]{Corollary}
\newtheorem{lemma}[theorem]{Lemma}
\newtheorem{proposition}[theorem]{Proposition}

\theoremstyle{definition}

\newcommand{\ep}{\epsilon}

\newcommand{\md}{\mathcal{D}}

\newcommand{\cc}{\mathbb{C}}
\newcommand{\cov}{\mathrm{Cov}}

\newcommand{\ee}{\mathbb{E}}

\newcommand{\mf}{\mathcal{F}}

\newcommand{\ml}{\mathcal{L}}

\newcommand{\pp}{\mathbb{P}}

\newcommand{\ra}{\rightarrow}
\newcommand{\rr}{\mathbb{R}}

\newcommand{\var}{\mathrm{Var}}

\newcommand{\vp}{\varphi^\prime}

\newcommand{\xp}{X^\prime}

\newcommand{\xx}{\mathcal{X}}

\newcommand{\zz}{\mathbb{Z}}

\title[A short survey of Stein's method]{A short survey of Stein's method}

\author[Sourav Chatterjee]
{Sourav Chatterjee\thanks{The author was partially supported by NSF grant DMS-1309618 during the preparation of this article.}}

\begin{document}

\begin{abstract}
Stein's method is a powerful technique for proving central limit theorems in probability theory when more straightforward approaches cannot be implemented easily. This article begins with a survey of the historical development of Stein's method and some recent advances. This is followed by a description of a ``general purpose'' variant of Stein's method that may be called the generalized perturbative   approach, and an application of this method to minimal spanning trees.  The article concludes with the descriptions   of some well known open problems that may possibly be solved by the perturbative approach or some other variant of Stein's method.
\end{abstract}



\maketitle


\section{Introduction}
A sequence of real-valued random variables $Z_n$ is said to converge in distribution to a limiting random variable $Z$ if 
\[
\lim_{n\ra\infty} \pp(Z_n\le t) = \pp(Z\le t)
\]
at all $t$ where the map $t\mapsto \pp(Z\le t)$ is continuous. It is equivalent to saying that 
for all bounded  continuous functions $g$ from $\rr$ into $\rr$ (or into $\cc$), 
\begin{equation}\label{conveq}
\lim_{n\to \infty} \ee g(Z_n) = \ee g(Z)\, .
\end{equation}
Often, it is not necessary to consider all bounded continuous $g$, but only $g$ belonging to a smaller class. For example, it suffices to consider all $g$ of the form $g(x)=e^{itx}$, where $i=\sqrt{-1}$ and $t\in \rr$ is arbitrary, leading to the method of characteristic functions (that is, Fourier transforms) for proving convergence in distribution. 

The case where $Z$ is a normal (alternatively, Gaussian) random variable is of particular interest to probabilists and statisticians, because of the frequency of its appearance as a limit in numerous problems. The normal distribution with mean $\mu$ and variance $\sigma$ is the probability distribution on $\rr$ that has probability density
\[
\frac{1}{\sigma\sqrt{2\pi}}e^{-(x-\mu)^2/2\sigma^2}
\] 
with respect to Lebesgue measure. The case $\mu=0$ and $\sigma=1$ is called ``standard normal'' or ``standard Gaussian''. To show that a sequence of random variables $Z_n$ converges in distribution to this $Z$, one simply has to show that for each $t$,
\[
\lim_{n\ra\rr} \ee(e^{itZ_n}) = \ee(e^{itZ}) = e^{it\mu - \sigma^2 t^2/2}\, .
\]
Indeed, this is the most well known approach to proving the classical central limit theorem for sums of independent random variables. 

Besides characteristic functions, there are two other classical approaches to proving central limit theorems. First, there is the method of moments, which involves showing that $\lim_{n\ra\infty}\ee(Z_n^k) = \ee(Z^k)$ for every positive integer $k$. 
Second, there is an old technique of Lindeberg \cite{lindeberg22}, which  has recently regained prominence. I will explain Lindeberg's method in Section \ref{open}. 

In 1972, Charles Stein \cite{stein72} proposed a radically different approach to proving convergence to normality. Stein's observation was that the standard normal distribution is the only probability distribution that satisfies the equation
\begin{equation}\label{stein}
\ee (Zf(Z)) = \ee f'(Z)
\end{equation}
for all absolutely continuous $f$ with a.e.~derivative $f'$ such that $\ee|f'(Z)|<\infty$. From this, one might expect that if $W$ is a random variable that satisfies the above equation in an approximate sense, then the distribution of $W$ should be close to the standard normal distribution. Stein's approach to making this idea precise was as follows. 

Take any bounded measurable function $g:\rr\ra\rr$. Let $f$ be a bounded solution of the differential equation
\begin{equation}\label{steineq}
f'(x)-xf(x) = g(x) - \ee g(Z)\, ,
\end{equation}
where $Z$ is a standard normal random variable. Stein \cite{stein72} showed that a bounded solution always exists, and therefore for any random variable $W$,
\[
\ee g(W) - \ee g(Z) = \ee (f'(W) - Wf(W))\, .
\]
If the right-hand side is close to zero, so is the left. If we want to consider the supremum of the left-hand side over a class of functions $g$, then it suffices to do the same on the right for all $f$ obtained from such $g$. For example, one can prove the following simple proposition: 
\begin{proposition}\label{mainprop}
Let $\md$ be the set of all $f:\rr \to \rr$ that are twice continuously differentiable, and $|f(x)|\le 1$, $|f'(x)|\le1$ and $|f''(x)|\le 1$ for all $x\in \rr$. Let $Z$ be a standard normal random variable and $W$ be any random variable. Then
\[
\sup_{t\in \rr} |\pp(W\le t) - \pp(Z\le t)|\le 2\Big(\sup_{f\in \md} |\ee(f'(W)-W f(W))|\Big)^{1/2}\, .
\]
\end{proposition}
\begin{proof}
Fix $\ep > 0$. Let $g(x) = 1$ if $x\le t$ and $0$ if $x\ge t+\ep$, with linear interpolation in the interval $[t,t+\ep]$. Let $f$ be a solution of the differential equation  \eqref{steineq}. By standard estimates \cite[Lemma 2.4]{chen}, $|f(x)|\le 2/\ep$, $|f'(x)|\le \sqrt{2/\pi}/ \ep$ and $|f''(x)|\le 2/\ep$ for all $x$. Consequently, $(\ep /2)f\in \md$.  Since the probability density function of $Z$ is bounded by $1/\sqrt{2\pi}$ everywhere, it follows that
\begin{align*}
\pp(W\le t) &\le \ee g(W) \\
&= \ee g(Z) + \ee(f'(W)-Wf(W))\\
&\le \pp(Z\le t) + \frac{\ep}{\sqrt{2\pi}} + \ee(f'(W)-Wf(W))\\
&\le  \pp(Z\le t) + \frac{\ep}{\sqrt{2\pi}} + \frac{2}{\ep}\sup_{h\in \md} \ee(h'(W)-Wh(W))\, .
\end{align*}
Similarly, taking $g(x)=1$ if $x\le t-\ep$, $g(x)=0$ if $x\ge t$ and linear interpolation in the interval $[t-\ep, t]$, we get 
\begin{align*}
\pp(W\le t) &\ge  \pp(Z\le t) - \frac{\ep}{\sqrt{2\pi}} - \frac{2}{\ep}\sup_{h\in \md} |\ee(h'(W)-Wh(W))|\, .
\end{align*}
The proof of the proposition is now easily completed  by optimizing over $\ep$. 
\end{proof}
The convenience of dealing with the right-hand side in Proposition \ref{mainprop} is that it involves only one random variable, $W$, instead of the two variables $W$ and $Z$ that occur on the left. This simple yet profound idea  gave birth to the field of Stein's method, that has survived the test of time and is still alive as an active field of research within probability theory after forty years of its inception. 

\section{A brief history of Stein's method}
Stein introduced his method of normal approximation in the seminal paper \cite{stein72} in 1972. The key to Stein's implementation of his idea was the method of exchangeable pairs, devised by Stein in \cite{stein72}. The key idea is as follows. A pair of random variables or vectors $(W,W')$ is called an {\it exchangeable pair} if $(W,W')$ has the same distribution as $(W', W)$. Stein's basic idea was that if $(W,W')$ is an exchangeable pair such that for some small number $\lambda$,
\begin{align*}
&\ee(W'-W\mid W) = -\lambda W  + o(\lambda)\, ,\\
&\ee((W'-W)^2\mid W) = 2\lambda + o(\lambda)\, , \text{ and }\\
&\ee|W'-W|^3 = o(\lambda)\, ,
\end{align*}
where $o(\lambda)$ denotes random or nonrandom quantities that have typical magnitude much smaller than $\lambda$, then $X$ is approximately standard normal. Without going into the precise details, Stein's reasoning goes like this: Given any $f\in \md$ where $\md$ is the function class from Proposition \ref{mainprop}, it follows by exchangeability that
\begin{align*}
\ee((W'-W)(f(W')+f(W))) = 0\, ,
\end{align*}
because the left-hand side is unchanged if $W$ and $W'$ are exchanged, but it also becomes the negation of itself. But note that by the given conditions,
\begin{align*}
\frac{1}{2\lambda} \ee((W'-W)(f(W')+f(W))) &= \frac{1}{2\lambda} \ee((W'-W)(f(W')-f(W)))\\
&\qquad  + \frac{1}{\lambda} \ee((W'-W)f(W))\\
&= \frac{1}{2\lambda} \ee((W'-W)^2 f'(W)) - \ee(Wf(W)) + o(1)\\
&= \ee(f'(W)) - \ee(Wf(W)) + o(1)\, ,
\end{align*}
where $o(1)$ denotes a small quantity. 

For example, if $W = n^{-1/2}(X_1+\cdots+X_n)$ for i.i.d.\ random variables $X_1,\ldots,X_n$ with mean zero, variance one and $\ee|X_1|^3 <\infty$, then taking 
\[
W' = W - \frac{X_I}{\sqrt{n}} + \frac{X_I'}{\sqrt{n}}\, ,
\]
where $I$ is uniformly chosen from $\{1,\ldots,n\}$ and for each $i$, $X_i'$ is an independent random variable having the same distribution as $X_i$, we get an exchangeable pair that satisfies the three criteria listed above with $\lambda = 1/n$ (easy to check).

The monograph \cite{stein86} also contains the following abstract generalization of the above idea. Suppose that we have two random variables $W$ and $Z$, and suppose that $T_0$ is an operator on the space of bounded measurable functions such that $\ee T_0f(Z) = 0$ for all $f$. Let $\alpha$ be any  map that takes a bounded measurable function $f$ on $\rr$ to an antisymmetric bounded measurable function $\alpha f$ on $\rr^2$ (meaning that $\alpha f (x,y)=-\alpha f(y,x)$ for all $x,y$). 

In the above setting, note that if $W'$ is a random variable such that $(W,W')$ is an exchangeable pair, then $\ee\alpha f(W, W') = 0$ for any $f$. For a function $h$ of two variables, let 
\[
Th(x) := \ee(h(W,W')\mid  W=x)\, ,
\]
so that $\ee T\alpha f(W) = \ee\alpha f(W,W') = 0$ for any $f$. Consequently, given $g$, if $f$ is a solution of the functional equation 
\[
T_0f(x) = g(x) - \ee g(Z)\, ,
\]
then
\begin{equation}\label{steinid}
\ee g(W) - \ee g(Z) = \ee T_0 f(W) = \ee (T_0 f(W) - T\alpha f(W))\, .
\end{equation}
Thus, if $T_0 \approx T\alpha$, then $Z$ and $W$ have approximately the same distributions. For example, for normal approximation, we can take $T_0f(x) = f'(x)-xf(x)$ and $\alpha f(x,y)=(2\lambda)^{-1}(x-y)(f(x)+f(y))$, where $\lambda$ is as above. If the three conditions listed by Stein hold for an exchangeable pair $(W,W')$, then indeed $T_0 \approx T\alpha$, as we have shown above. 

The identity \eqref{steinid} is the content of a famous commutative diagram of Stein \cite{stein86}. It has been used in contexts other than normal approximation --- for example, for Poisson approximation in \cite{cdm} and for the analysis of Markov chains in \cite{diaconis}. 

A notable success story of Stein's method was authored by Bolthausen~\cite{bolthausen} in 1984, when he used a sophisticated version of the method of exchangeable pairs to obtain an error bound in a famous combinatorial central limit theorem of Hoeffding. The problem here is to prove a central limit theorem for an object like $W = \sum_{i=1}^n a_{i\pi(i)}$, where $a_{ij}$ is a given array of real numbers, and $\pi$ is a uniform random permutation of $\{1,\ldots, n\}$. Bolthausen defined
\[
W' = W - a_{I\pi(I)} - a_{J\pi(J)} + a_{I\pi(J)} + a_{J \pi(I)}\, ,
\]
and proved that $(W,W')$ is an exchangeable pair satisfying the three required conditions. The difficult part in Bolthausen's work was to derive a sharp error bound, since the error rate given by a result like Proposition \ref{mainprop} is usually not optimal. 

Incidentally, it has been proved recently by R\"ollin \cite{rollin} that to apply exchangeable pairs for normal approximation, it is actually not necessary that $W$ and $W'$ are exchangeable; one can make an argument go through if $W$ and $W'$ have the same distribution. 

Stein's 1986 monograph \cite{stein86} was the first book-length treatment of Stein's method. After the publication of \cite{stein86}, the field was given a boost by the popularization of the method of dependency graphs by Baldi and Rinott \cite{baldi1}, a striking application to the number of local maxima of random functions by Baldi, Rinott and Stein \cite{baldi2}, and central limit theorems for random graphs by Barbour, Karo\'nski and Ruci\'nski \cite{barbour2}, all in 1989. 

The method of dependency graphs, as a version of Stein's method,  was introduced in Louis Chen's 1971 Ph.D.~thesis on Poisson approximation and the subsequent publication \cite{chen0}. It was developed further by Chen \cite{chen86} before being brought to wider attention by Baldi and Rinott~\cite{baldi1}. Briefly, the method may be described as follows. Suppose that $(X_i)_{i\in V}$ is a collection of random variables indexed by some finite set $V$. A {\it dependency graph} is an undirected graph on the vertex set $V$ such that if $A$ and $B$ are two subsets of $V$ such that there are no edges with one endpoint in  $A$ and the other in $B$, then the collections $(X_i)_{i\in A}$ and $(X_i)_{i\in B}$ are independent. Fix a dependency graph, and for each $i$, let $N_i$ be the neighborhood of $i$ in this graph, including the vertex $i$. Let $W = \sum_{i\in V} X_i$ and assume that $\ee(X_i)=0$ for each $i$. Define
\[
W_i := \sum_{j\not\in N_i} X_j\, ,
\]
so that $W_i$ is independent of $X_i$. 
Then note that for any smooth $f$,
\begin{align*}
\ee(W f(W)) &= \sum_{i\in V} \ee(X_i f(W)) \\
&= \sum_{i\in V} \ee(X_i (f(W)-f(W_i)))\\
&\approx \sum_{i\in V} \ee(X_i (W-W_i) f'(W)) =  \ee\biggl(\biggl(\sum_{i\in V} X_i (W-W_i)\biggr) f'(W)\biggr)\, ,
\end{align*}
where the approximation holds under the condition that $W\approx W_i$ for each $i$. Define $T := \sum_{i\in V} X_i(W-W_i)$. Let $\sigma^2 := \ee T$. The above approximation, when valid, implies that $\var W = \ee W^2 \approx \sigma^2$. Therefore if $T$ has a small variance, then $\ee (Wf(W)) \approx \sigma^2 \ee f'(W)$. By a slight variant of Proposition \ref{mainprop}, this shows that $W$ is approximately normal with mean zero and variance $\sigma^2$. 

To gain a hands-on understanding of the dependency graph method,  the reader can check that this technique works when $Y_1,\ldots, Y_n$ are independent random variables with mean zero, and $X_i = n^{-1/2} Y_i Y_{i+1}$ for $i=1,\ldots, n-1$. Here $V= \{1,\ldots, n-1\}$, and a dependency graph may be defined by putting an edge between $i$ and $j$ whenever $|i-j|=1$. 

The new surge of activity that began in the late eighties continued through the nineties, with important contributions coming from Barbour \cite{barbour0} in 1990, who introduced the diffusion approach to Stein's method; Avram and Bertsimas \cite{avram} in 1993, who applied Stein's method to solve an array of important problems in geometric probability; Goldstein and Rinott \cite{goldstein2} in 1996, who developed the method of size-biased couplings for Stein's method, improving on earlier insights of Baldi, Rinott and Stein \cite{baldi2}; Goldstein and Reinert \cite{goldstein1} in 1997, who introduced the method of zero-bias couplings; and Rinott and Rotar \cite{rinott} in 1997, who solved a well known open problem related to the antivoter model using Stein's method. Sometime later, in 2004, Chen and Shao \cite{chen2} did an in-depth study of the dependency graph approach, producing optimal Berry-Ess\'een type error bounds in a wide range of problems. The 2003 monograph of Penrose \cite{penrose} gave extensive applications of the dependency graph approach to problems in geometric probability.

I will now try to outline the basic concepts behind some of the methods cited in the preceding paragraph. 

The central idea behind Barbour's diffusion approach \cite{barbour0} is that if a probability measure $\mu$ on some abstract space is the unique invariant measure for a diffusion process with generator $\ml$, then under mild conditions $\mu$ is the only probability measure satisfying $\int \ml f d\mu = 0$ for all $f$ in the domain of $\ml$; therefore, if a probability measure $\nu$ has the property that $\int \ml fd\nu \approx 0$ in some suitable sense for a large class of $f$'s, then one may expect that $\nu$ is close to $\mu$ is some appropriate metric. Generalizing Stein's original approach, Barbour then proposed the following route to make this idea precise. Given a function $g$ on this abstract space, one can try to solve for
\[
\ml f(x) = g(x) - \int g d\mu\, ,
\]
and use 
\[
\int g d\nu - \int gd\mu = \int \ml f d\nu \approx 0\, .
\]
To see how Stein's method of normal approximation fits into this picture, one needs to recall that the standard normal distribution on $\rr$ is the unique invariant measure for a diffusion process known as the {\it Ornstein-Uhlenbeck process}, whose generator is $\ml f (x) = f''(x)-xf'(x)$. This looks different than the original Stein operator $f'(x)-xf(x)$, but it is essentially the same: one has to simply replace $f$ by $f'$ and $f'$ by $f''$. 

In \cite{barbour0}, Barbour used this variant of Stein's method to solve some problems about diffusion approximation. However, the most significant contribution of Barbour's paper was a clarification of the mysterious nature of the method of exchangeable pairs. A one dimensional diffusion process $(X_t)_{t\ge 0}$ with drift coefficient $a(x)$ and diffusion coefficient $b(x)$ is a continuous time stochastic process adapted to some filtration $\{\mf_t\}_{t\ge 0}$ satisfying, as $h \ra 0$, 
\begin{align*}
&\ee(X_{t+h} - X_t \mid \mf_t)  = a(X_t) h + o(h)\, ,\\
&\ee((X_{t+h}-X_t)^2 \mid \mf_t) = b(X_t)^2 h + o(h)\, , \text{ and}\\
&\ee|X_{t+h}-X_t|^3 = o(h)\, .
\end{align*}
An exchangeable pair $(W, W')$ naturally defines a stationary, reversible Markov chain $W_0,W_1,W_2,\ldots$, where $W_0=W$, $W_1=W'$, and for each $i$, the conditional distribution of $W_{i+1}$ given $W_i$ is the same as that of $W_1$ given $W_0$. If the pair $(W,W')$ satisfies the three conditions listed by Stein for some small $\lambda$, then in a scaling limit as $\lambda \ra 0$, the Markov chain defined above converges to a diffusion process with drift function $a(x)=-x$ and diffusion coefficient $\sqrt{2}$. This is precisely the standard Ornstein-Uhlenbeck process whose stationary distribution is the standard normal. Therefore one can expect that $W$ is approximately normally distributed. Note that this argument is quite general, and not restricted to normal approximation. In a later paragraph, I will briefly point out some  generalizations of Stein's method using Barbour's approach.

The method of size-biased couplings in Stein's method was introduced in the paper of Baldi, Rinott and Stein \cite{baldi2}, and was fully developed by Goldstein and Rinott \cite{goldstein2}. The size-biased transform of a non-negative random variable $W$ with mean $\lambda$ is a random variable, usually denoted by $W^*$, such that for all $g$,
\[
\ee(Wg(W)) = \lambda \ee g(W^*)\,.
\]
Size biasing is actually a map on probability measures, which takes a probability measure $\mu$ on the non-negative reals to a probability  measure $\nu$ defined as $d\nu(x) = \lambda^{-1}xd\mu(x)$, where $\lambda$ is the mean of $\mu$. Size biasing is an old concept, predating Stein's method, probably originating in the survey sampling literature. (Actually, the name ``size-biasing'' comes from the survey sampling procedure where a sample point is chosen with probability proportional to some notion of size.) As a consequence of its classical origins and usefulness in a variety of domains, there are many standard procedures to construct size-biased versions of complicated random variables starting from simpler ones. For example, if $X_1,\ldots, X_n$ are i.i.d.\ non-negative random variables, and $W= X_1+\cdots+X_n$, and $X_1^*$ is a size-biased version of $X_1$, then $W^* = X_1^* +X_2+\cdots +X_n$ is a size-biased version of $W$. To see this, just note that for any $g$,
\begin{align*}
\ee(Wg(W)) &= n \ee(X_1 g(X_1+\cdots X_n)) \\
&= n\ee(X_1)\ee g(X_1^* + X_2+\cdots+X_n)\\
&= \ee(W) \ee g(W^*)\, .
\end{align*}
For more complicated examples, see \cite{goldstein2}.

 In Stein's method, size biasing is used in the following manner: Suppose that $W$ is  a non-negative random variable with mean $\lambda$ and variance $\sigma^2$. Suppose that we are able to construct a size-biased version $W^*$ of $W$ on the same probability space, such that 
\begin{align*}
& \ee(W^*-W \mid W) = \frac{\sigma^2}{\lambda}(1+o(1))\, , \text{ and } \\
& \ee(W^* - W)^2 = o\biggl(\frac{\sigma^{3}}{\lambda}\biggr)\, .
\end{align*}
Then the standardized random variable $X := (W-\lambda)/\sigma$ is approximately standard normal. To understand why this works, let $Y:= (W^*-\lambda)/\sigma$ and note that under the two conditions displayed above,
\begin{align*}
\ee(Xf(X)) &= \frac{1}{\sigma} \ee (W f(X)) - \frac{\lambda}{\sigma} \ee f(X) \\
&= \frac{\lambda}{\sigma} \ee (f(Y)- f(X))\\
&= \frac{\lambda}{\sigma} \ee ((Y-X) f'(X)) + \frac{\lambda}{\sigma} O(\ee(Y-X)^2)\\
&= \frac{\lambda}{\sigma^2} \ee(\ee(W^*-W|W)f'(X)) + \frac{\lambda}{\sigma^3} O(\ee(W^*-W)^2)\\
&= \ee f'(X)  + o(1)\, .
\end{align*}
For a mathematically precise version of the above argument, see \cite[Theorem 1.1]{goldstein2}.

The method of size biased couplings is quite a powerful tool for proving central limit theorems for non-negative random variables, especially those that arise as sums of mildly dependent variables. The only hurdle is that one has to be able to construct a suitable size-biased coupling. There is also the other limitation that $W$ has to be non-negative.  To overcome these limitations, Goldstein and Reinert~\cite{goldstein1} introduced the method of zero-bias couplings. Given a random variable $W$ with mean zero and variance $\sigma^2$, the zero-biased transform $W'$ of $W$ is a random variable satisfying
\[
\ee (W f(W)) = \sigma^2 \ee f'(W')
\]
for all differentiable $f$ whenever the left-hand side is well-defined. It is clear from Proposition \ref{mainprop} that if one can define a zero-bias transform $W'$ on the same probability space as $W$ such that $W'\approx W$ with high probability, then $W$ is approximately normal with mean $0$ and variance $\sigma^2$. The construction of zero-bias transforms can be quite tricky. The method has been systematically developed and used to solve a variety of problems by a number of authors, starting with Goldstein and Reinert \cite{goldstein1}.  

A feature of Stein's method of normal approximation that has limited its applicability throughout the history of the subject is that it works only for problems where ``something nice'' happens. This is true of all classical versions of the method, such as the method of exchangeable pairs, the dependency graph approach, size-biased couplings and  zero-bias couplings. For exchangeable pairs, we need that the three conditions listed by Stein are valid. For dependency graphs, we need the presence of a dependency graph of relatively  small degree. For the coupling techniques, we need to be able to construct the couplings. Given a general problem with no special structure, it is often difficult to make these methods work.  Intending to come up with a more general approach, I introduced a new method in 2008 in the paper \cite{chatterjee3} for discrete systems, and a corresponding continuous version in \cite{chatterjee4} in 2009. This new approach (which I am calling  {\it the generalized perturbative approach} in this article) was used to solve a number of questions in geometric probability in \cite{chatterjee3}, random matrix central limit theorems in~\cite{chatterjee4}, number theoretic central limit theorems in~\cite{chatterjeesound}, and  an error bound in a  central limit theorem for minimal spanning trees in~\cite{chatterjeesen}. The generalized perturbative method is described in detail in Section~\ref{mainresult}.

 The paper \cite{chatterjee4} also introduced the notion of {\it second order Poincar\'e inequalities}. The simplest second order Poincar\'e inequality, derived in \cite{chatterjee4}, states that if $X = (X_1,\ldots, X_n)$ is a vector of i.i.d.\ standard normal random variables, $f:\rr^n \ra \rr$ is a twice continuously differentiable function with gradient $\nabla f$ and Hessian matrix $\mathrm{Hess} f$, and $W := f(X)$ has mean zero and variance $1$, then 
\[
\sup_{A\in \mathcal{B}(\rr)}|\pp(W\in A)-\pp(Z\in A)| \le 2\sqrt{5} (\ee\|\nabla f(X)\|^4)^{1/4} (\ee\|\mathrm{Hess} f(X)\|_{\mathrm{op}}^4)^{1/4}\, ,
\] 
where $\|\nabla f(X)\|$ is the Euclidean norm of $\nabla f(X)$, $\|\mathrm{Hess} f(X)\|_{\mathrm{op}}$ is the operator norm of $\mathrm{Hess} f(X)$, and $\mathcal{B}(\rr)$ is the set of Borel subsets of $\rr$. In \cite{chatterjee4}, this inequality was used to prove new central limit theorems for linear statistics of eigenvalues of random matrices. The name ``second order Poincar\'e inequality'' is inspired from the analogy with the usual Poincar\'e inequality for the normal distribution, which states that $\var f(X) \le \ee\|\nabla f(X)\|^2$ for any absolutely continuous $f$. Although this does not look like anything related to Stein's method, a close inspection of the proof in \cite{chatterjee4} makes it clear that it is in fact an offshoot of Stein's method.

Incidentally, the usual Poincar\'e inequality has also been used to prove central limi theorems, for example by Chen \cite{chen88}, using a characterization of the normal distribution by Borovkov and Utev \cite{borovkov}. 

Second order Poincar\'e inequalities have been useful in several subsequent works, e.g.~in Nourdin, Peccati and Reinert \cite{nourdin}, Nolen \cite{nolen}, etc. Indeed, it may be said that the whole thriving area of Stein's method in Malliavin calculus, pioneered by Nourdin and Peccati \cite{nourdinpeccati}, is an ``abstractification'' of the ideas contained in \cite{chatterjee3} and \cite{chatterjee4}. The new method was later unified with other branches of Stein's method through the concept of {\it Stein couplings} introduced by Chen and R\"ollin \cite{chenrollin}.

Normal approximation is not the only area covered by Stein's method. In 1975, Louis Chen \cite{chen0} devised a version of Stein's method for Poisson approximation, expanding on his 1971 Ph.D.~thesis under Stein. The {\it Chen-Stein method} of Poisson approximation is a very useful tool in its own right, finding applications in many areas of the applied sciences. The main idea is that a Poisson random variable $X$ with mean $\lambda$ is the only kind of random variable satisfying 
\[
\ee (X f(X)) = \lambda \ee f(X+1)
\]
for every $f$, and then proceed from there as usual by developing a suitable version of Proposition \ref{mainprop}. The subject of Poisson approximation by Stein's method took off with the papers of Arratia, Goldstein and Gordon \cite{arratia, arratia2} and the classic text of Barbour, Holst and Janson \cite{barbour1}, all appearing in the period between 1989 and 1992. A relatively recent survey of Poisson approximation by Stein's method is given in my paper \cite{cdm} with Diaconis and Meckes.

Besides normal and Poisson, Stein's method has been used sometimes for other kinds of distributional approximations. One basic idea was already available in Stein's 1986 monograph \cite{stein86}, and a different one in Barbour's paper \cite{barbour0} on the diffusion approach to Stein's method. These ideas  were implemented in various forms by Mann \cite{mann} in 1994 for chi-square approximation, Luk \cite{luk} in 1997 for gamma approximation, Holmes \cite{holmes} in 2004 for birth-and-death chains, and Reinert \cite{reinert0} in 2005 for approximation of general densities. In 2005, Fulman \cite{fulman05} extended the method of exchangeable pairs to study Plancherel measures on symmetric groups. Stein's method for a mixture of two normal distributions, with an application to spin glasses, appeared in my 2010 paper \cite{chatterjee5}, while another non-normal distribution arising at the critical temperature of the Curie-Weiss model of ferromagnets was tackled in my joint paper with Shao \cite{chatterjeeshao} in 2011 and in a paper of Eichelsbacher and L\"owe \cite{eich} in 2010. Several papers on Stein's method for geometric and exponential approximations have appeared in the literature, including an early paper of Pek\"oz \cite{pekoz1} from 1996, a paper of myself with Fulman and R\"ollin \cite{cfr} that appeared in 2011, and papers of Pek\"oz and R\"ollin \cite{pekoz2} and Pek\"oz, R\"ollin and Ross \cite{pekoz3} that appeared in 2011 and 2013 respectively. 

Another area of active research is Stein's method for multivariate normal approximation. Successful implementations were carried out by G\"otze \cite{gotze} in 1991, Bolthausen and G\"otze \cite{bolthausen2} in 1993, and Rinott and Rotar \cite{rinottrotar96} in 1996. The complexities of G\"otze's method were clarified by Bhattacharya and Holmes \cite{bh} in 2010. In a joint paper \cite{cm} with Meckes in 2008, we found a way to implement the method of exchangeable pairs in the multivariate setting. The main idea here is to generalize Barbour's diffusion approach to the multidimensional setting, by considering the multivariate Ornstein-Uhlenbeck process and the related semigroup. This naturally suggests a multivariate generalization of the three exchangeable pair conditions listed by Stein. The relevant generalization of the Stein equation \eqref{steineq}, therefore, is
\[
\Delta f(x) - x\cdot \nabla f(x) = g(x) - \ee g(Z)\, ,
\]
where $\Delta f$ is the Laplacian of $f$, $\nabla f$ is the gradient of $f$, $x\cdot \nabla f(x)$ is the inner product of the vector $x$ and the gradient vector $\nabla f (x)$, and $Z$ is a multidimensional standard normal random vector. The method was greatly advanced, with many applications, by Reinert and R\"ollin \cite{reinert, reinert2} in 2009 and 2010. Further advances were made in the recent manuscript of R\"ollin \cite{rollin2}.

Incidentally, there is a rich classical area of multivariate normal approximation, and a lot of energy spent on what class of sets the approximation holds for. This remains to be worked out for Stein's method.

Besides distributional approximations, Stein's method has also been used to prove concentration inequalities. Preliminary attempts towards deviation inequalities were made by Stein in his 1986 monograph \cite{stein86}, which were somewhat taken forward by Rai\v{c} in 2007. The first widely applicable set of concentration inequalities using Stein's method of exchangeable pairs appeared in my Ph.D.~thesis \cite{chatterjee1} in 2005, some of which were collected together in the 2007 paper \cite{chatterjee2}. A more complex set of examples was worked out in a  later paper with Dey \cite{chatterjeedey} in 2010. One of the main results of \cite{chatterjee1, chatterjee2} is that if $(W,W')$ is an exchangeable pair of random variables and $F(W,W')$ is an antisymmetric function of $(W,W')$ (meaning that $F(W,W')=-F(W',W)$), then for all $t\ge 0$,
\[
\pp(|f(W)|\ge t) \le 2e^{-t^2/2C}\, ,
\]
where $f(W)=\ee(F(W,W')|W)$ and $C$ is a number such that  
\[
|(f(W)-f(W'))F(W,W')|\le C \ \ \text{with probability one.}
\]
Surprisingly, this abstract machinery has found quite a bit of use in real applications. In 2012, Mackey and coauthors \cite{mackey} extended the method to the domain of matrix concentration inequalities, thereby solving some problems in theoretical machine learning. In 2011, Ghosh and Goldstein \cite{ghosh1, ghosh2} figured out a way to use size-biased couplings for concentration inequalities.

There are a number of nonstandard applications of Stein's method that have not yet gathered a lot of follow up action, for example, Edgeworth expansions (Rinott and Rotar \cite{rr}), rates of convergence of Markov chains (Diaconis \cite{diaconis}), strong approximation in the style of the KMT embedding theorem (my paper \cite{chatterjee6}), moderate deviations (Chen et al.~\cite{chenfang}) and even in the analysis of simulations (Stein et al.~\cite{steinetal}). A great deal of hard work has gone into proving sharp Berry-Ess\'een bounds using Stein's method. Some of this literature is surveyed in Chen and Shao~\cite{chen2}.

A number of well written monographs dedicated to various aspects of Stein's method are in existence. The book of Barbour, Holst and Janson \cite{barbour1} is a classic text on Poisson approximation by Stein's method. The recent monograph by Chen, Goldstein and Shao \cite{chen} is a very readable and comprehensive account of normal approximation by Stein's method. The survey of Ross \cite{ross}, covering many aspects of Stein's method, is already attaining the status of a must-read in this area. The monograph \cite{npbook} of Nourdin and Peccati describes the applications of Stein's method in Malliavin calculus. The edited volumes \cite{barbourchen} and \cite{diaconisholmes} are also worth a look. 

Lastly, I should clarify that the above review was an attempt to cover only the theoretical advances in Stein's method. The method has found many applications in statistics, engineering, machine learning, and other areas of applications of mathematics. I have made no attempt to survey these applications.

This concludes my very rapid survey of existing techniques and ideas in Stein's method. I apologize to anyone whose work I may have inadvertently left out. In the rest of this manuscript, I will attempt to briefly explain the generalized perturbative method introduced in the papers \cite{chatterjee3} and~\cite{chatterjee4}, and then conclude by stating some open problems.

\section{The generalized perturbative approach}\label{mainresult}
Let $\xx$ be a measure space and suppose $X= (X_1,\ldots,X_n)$ is a vector of independent $\xx$-valued random variables. Let $f:\xx^n \ra \rr$ be a measurable function and let $W := f(X)$. Suppose that $\ee W = 0$ and $\ee W^2 = 1$. I will now outline a general technique for getting an upper bound on the distance of $W$ from the standard normal distribution using information about how $f$ changes when one coordinate of $X$ is perturbed. Such techniques have long been commonplace in the field of concentration inequalities. Suitable versions were introduced for the first time in the context of normal approximation in the papers \cite{chatterjee3, chatterjee4}. I am now calling this the {\it generalized perturbative approach} to Stein's method. The word ``generalized'' is added to the name because the method of exchangeable pairs is also a perturbative approach, but this is more general. 

Let $\xp = (\xp_1,\ldots,\xp_n)$ be an independent copy of $X$. Let $[n] = \{1,\ldots, n\}$, and for each $A \subseteq [n]$, define the random vector $X^A$ as
\[
X^A_i =
\begin{cases}
\xp_i &\text{ if } i\in A,\\
X_i &\text{ if } i\not \in A.
\end{cases}
\]
When $A$ is singleton set like $\{i\}$, write $X^i$ instead of $X^{\{i\}}$. Similarly, write $A\cup i$ instead of $A\cup \{i\}$.  Define a randomized derivative of $f$ along the $i$th coordinate~as
\[
\Delta_i f := f(X)-f(X^{i})\, ,
\]
and for each $A\subseteq [n]$ and $i\not\in A$, let
\[
\Delta_i f^A := f(X^A) - f(X^{A\cup i})\,.
\]
For each proper subset $A$ of $[n]$ define
\[
\nu(A) := \frac{1}{n{n-1\choose |A|}}\, .
\]
Note that when restricted to the set of all subsets of $[n]\backslash\{i\}$ for some given $i$, $\nu$ is a probability measure. Define
\[
T := \frac{1}{2}\sum_{i=1}^n\sum_{A\subseteq [n]\backslash\{i\}} \nu(A) \Delta_i f \Delta_i f^A\, .
\]
The generalized perturbative approach is based on the following completely general upper bound on the distance of $W$ from normality using the properties of the discrete derivatives $\Delta_i f$ and $\Delta_i f^A$. 
\begin{theorem}[Variant of Theorem 2.2 in \cite{chatterjee3}]\label{mainthm}
Let $W$ be as above and $Z$ be a standard normal random variable. Then 
\[
\sup_{t\in \rr}|\pp(W\le t) - \pp(Z\le t)|\le 2\biggl(\sqrt{\var (\ee(T|W))} + \frac{1}{4}\sum_{i=1}^n \ee|\Delta_i f|^3\biggr)^{1/2}\, .
\]
\end{theorem}
In practice, the variance of $\ee(T|W)$ may be upper bounded by the variance of  $\ee(T|X)$ or the variance of $T$, which are easier to handle mathematically. 

The following simple corollary may often be useful for problems with local dependence. We will see an application of this to minimal spanning trees in Section~\ref{mstsec}. 
\begin{corollary}\label{maincor}
Consider the setting of Theorem \ref{mainthm}. For each $i,j$, let $c_{ij}$ be a constant such that for all $A\subseteq [n]\backslash\{i\}$ and $B\subseteq [n]\backslash\{j\}$,
\[
\cov(\Delta_i f \Delta_i f^A,\,  \Delta_j f \Delta_j f^B) \le c_{ij}\, .
\]
Then 
\[
\sup_{t\in \rr}|\pp(W\le t) - \pp(Z\le t)|\le \sqrt{2}\biggl(\sum_{i,j=1}^n c_{ij}\biggr)^{1/4} + \biggl(\sum_{i=1}^n \ee|\Delta_i f|^3\biggr)^{1/2}\, .
\]
\end{corollary}
Intuitively, the above corollary says that if most pairs of discrete derivatives are approximately independent, then $W$ is approximately normal. This condition may be called {\it the approximate independence of small perturbations.} 

For example, if $X_1,\ldots, X_n$ are real-valued with mean zero and variance one, and $W= n^{-1/2}\sum X_i$, then we may take  $c_{ij} = 0$ when $i\ne j$ and $c_{ii} = C/n^2$ for some constant $C$ depending on the distribution of the $X_i$'s. Moreover note that $|\Delta_i f|$ is of order $n^{-1/2}$. Therefore, Corollary \ref{maincor}  gives a proof of the ordinary central limit theorem for sums of i.i.d.~random variables with an $n^{-1/4}$ rate of convergence. This rate is suboptimal, but this suboptimality is a general feature Stein's method, requiring quite a bit of effort to overcome. 


Theorem \ref{mainthm} was used to solve several questions in geometric probability (related to nearest neighbor distances and applications in statistics) in \cite{chatterjee3}, prove a number theoretic central limit theorem in \cite{chatterjeesound} and obtain a rate of convergence in a central limit theorem for minimal spanning trees in \cite{chatterjeesen}. When $X_1,\ldots, X_n$ are i.i.d.\ normal random variables, a ``continuous'' version of this theorem, where the perturbations are done in a continuous manner instead of replacing by independent copies, was proved in \cite{chatterjee4}. This continuous version of Theorem~\ref{mainthm} was then used to derive the so-called second order Poincar\'e inequality for the Gaussian distribution.

The remainder of this section is devoted to the proofs of Theorem \ref{mainthm} and Corollary~\ref{maincor}. 
Applications are worked out in the subsequent sections. 

\begin{proof}[Proof of Theorem \ref{mainthm}]
Consider the sum
\begin{align*}
\sum_{i=1}^n \sum_{A\subseteq[n]\backslash\{i\}} \nu(A) \Delta_if^A\, .
\end{align*}
Clearly, this is a linear combination of $\{f(X^A), A\subseteq [n]\}$.
It is a matter of simple verification that the positive and negative coefficients of $f(X^A)$ in this linear combination cancel out except when $A=[n]$ or $A=\emptyset$. In fact, the above expression is identically equal to $f(X)-f(\xp)$.

Let $g:\xx\to \rr$ be another measurable function. Fix $A$ and $i\not\in A$, and let $U = g(X)\Delta_if^A$. Then $U$ is a function of the random vectors $X$ and $\xp$. The joint distribution of $(X,\xp)$ remains unchanged if we interchange $X_i$ and $\xp_i$. Under this operation, $U$ changes to $U^\prime := -g(X^i)\Delta_if^A$. Thus, 
\begin{align*}
&\ee(U)=\ee(U^\prime) = \frac{1}{2}\ee(U + U^\prime)= \frac{1}{2}\ee\bigl(\Delta_ig\Delta_if^A\bigr)\, .
\end{align*}
As a consequence of the above steps and the assumption that $\ee W = 0$, we arrive at the identity 
\begin{align}
\ee(g(X)W) &= \ee\bigl(g(X)(f(X)-f(\xp))\bigr)\nonumber\\
&= \ee\biggl(\sum_{i=1}^n \sum_{A\subseteq [n]\backslash\{i\}} \nu(A) g(X) \Delta_i f\biggr)\nonumber \\
&= \frac{1}{2}\ee\biggl(\sum_{i=1}^n \sum_{A\subseteq [n]\backslash\{i\}} \nu(A) \Delta_i g \Delta_i f^A\biggr)\, .\nonumber 
\end{align}
In particular, taking $g=f$ gives  $\ee T = \ee W^2 = 1$. Next, take any $\varphi:\rr\ra\rr$ that belongs to the class $\md$ defined in Proposition \ref{mainprop}. Let $g := \varphi \circ f$. 
By the above identity, 
\begin{align*}
\ee(\varphi(W) W) &= \frac{1}{2} \sum_{i=1}^n \sum_{A\subseteq [n]\backslash\{i\}} \nu(A) \ee(\Delta_i g \Delta_i f^A).
\end{align*} 
By the mean value theorem and the fact that $|\varphi''(x)|\le 1$ for all $x$, 
\begin{align*}
\ee|\Delta_i g \Delta_i f^A - \vp(W) \Delta_i f\Delta_i f^A| &\le  \frac{1}{2}\ee|(\Delta_if)^2\Delta_if^A|\le \frac{1}{2}\ee|\Delta_if|^3\, ,
\end{align*}
where the last step follows by H\"older's inequality. Combining the last two displays gives 
\begin{align*}
|\ee(\varphi(W)W) - \ee(\vp(W) T)| &\le \frac{1}{4}\sum_{i=1}^n \sum_{A\subseteq [n]\backslash\{i\}} \nu(A) \ee|\Delta_i f|^3 = \frac{1}{4}\sum_{i=1}^n \ee|\Delta_i f|^3\, .
\end{align*}
Next, note that since $\ee T = 1$ and $|\vp(x)|\le 1$ for all $x$, 
\begin{align*}
|\ee(\vp(W) T) - \ee\vp(W)| &= |\ee(\vp(W)(\ee(T|W) - 1))|\\
&\le \ee|\ee(T| W)-1| \le \sqrt{\var (\ee(T|W))}\, .
\end{align*}
By the last two displays, 
\[
|\ee(\varphi(W)W - \vp(W))|\le \sqrt{\var (\ee(T|W))} + \frac{1}{4}\sum_{i=1}^n \ee|\Delta_i f|^3\, .
\]
Since this is true for any $\varphi \in \md$, Proposition \ref{mainprop} completes the proof of Theorem~\ref{mainthm}.
\end{proof}
\begin{proof}[Proof of Corollary \ref{maincor}]
Observe that
\begin{align*}
\var T &\le \frac{1}{4}\sum_{i,j=1}^n \sum_{\substack{A\subseteq [n]\backslash\{i\}\\B\subseteq [n]\backslash\{j\}} }\nu(A) \nu(B)\, \cov(\Delta_i f \Delta_i f^A, \, \Delta_j f \Delta_j f^B) \\
&\le \frac{1}{4}\sum_{i,j=1}^n \sum_{\substack{A\subseteq [n]\backslash\{i\}\\B\subseteq [n]\backslash\{j\}} }\nu(A) \nu(B)\, c_{ij} = \frac{1}{4}\sum_{i,j=1}^n c_{ij}\, .
\end{align*}
To complete the proof, apply Theorem \ref{mainthm} and the inequality $(x+y)^{1/2}\le x^{1/2}+y^{1/2}$ to separate out the two terms in the error bound. 
\end{proof}

\section{Application to minimal spanning trees}\label{mstsec}
In this section, I will describe an application of the generalized perturbative method to prove a central limit theorem for minimal spanning trees on lattices with random edge weights. This is a small subset of a joint work with Sen \cite{chatterjeesen}. The major objective of~\cite{chatterjeesen} was to obtain a rate of convergence, using the generalized perturbative approach, in a central limit theorem for the Euclidean minimal spanning tree due to Kesten and Lee \cite{kesten}.  Kesten and Lee  used the martingale central limit theorem to solve this problem (without an error bound), which was a long-standing open question at the time of its solution (except for the two-dimensional case, which was solved by Alexander \cite{alexander}). My interest in this area stemmed from a quest to understand normal approximation in random combinatorial optimization. Many such problems are still wide open. I will talk about some of them in the next section. 

Let $E$ be the set of edges of the integer lattice $\zz^d$. Let $(\omega_e)_{e\in E}$ be a set of i.i.d.\ edge weights, drawn from a continuous probability distribution on the positive real numbers with bounded support. For each $n$, let $V_n$ be the set $[-n,n]^d \cap \zz^d$, and let $E_n$ be the set of edges of $V_n$. The {\it minimal spanning tree} on the graph $G_n = (V_n, E_n)$ with edge weights $(\omega_e)_{e\in E_n}$ is the spanning tree that minimizes the sum of edge weights. Since the edge-weight distribution is continuous, this tree is unique. 

Let $M_n$ be the sum of edge weights of the minimal spanning tree on $G_n$. We will now see how to use Corollary \ref{maincor} to give a simple proof of the following central limit theorem for $M_n$. 
\begin{theorem}[Corollary of Theorem 2.4 in \cite{chatterjeesen}]\label{mst}
Let $\mu_n := \ee M_n$, $\sigma_n^2 := \var M_n$, and 
\[
f_n = f_n((\omega_e)_{e\in E_n}) := \frac{M_n - \mu_n}{\sigma_n}\, ,
\]
so that $f_n$ is a standardized version of $M_n$, with mean zero and variance one. Then $f_n$ converges in law to the standard normal distribution as $n$ goes to infinity.
\end{theorem}
Note that the above theorem does not have a rate of convergence. Theorem~2.4 in~\cite{chatterjeesen} has an explicit rate of convergence, but the derivation of that rate will take us too far afield; moreover that will be an unnecessary digression from the main purpose of this section, which is to demonstrate a  nontrivial application of the generalized perturbative approach. In the remainder of this section, I will present a short proof of Theorem \ref{mst} using the version of the generalized perturbative approach given in Corollary \ref{maincor}. 

To apply Corollary \ref{maincor}, we first have to understand how $M_n$ changes when one edge weight is replaced by an independent copy. This is a purely combinatorial issue. Following the notation of the previous section, I will denote the difference by $\Delta_e M_n$. The goal, eventually, is to show that $\Delta_e M_n$ is approximately equal to a quantity that depends only on some kind of a local neighborhood of $e$. This will allow us to conclude that the covariances in Corollary \ref{maincor} are small. The following lemma gives a useful formula for the  discrete derivative $\Delta_e M_n$, which is a first step towards this eventual goal.
\begin{lemma}\label{difflmm}
For each edge $e\in E$ and each $n$ such that $e\in E_n$, let $\alpha_{e,n}$ denote the smallest real number $\alpha$ such that there is a path from one endpoint of $e$ to the other,  lying entirely in $V_n$ but not containing the edge $e$, such that all edges on this path have weight $\le \alpha$. If the edge weight $\omega_e$ is replaced by an independent copy $\omega'_e$, and $\Delta_eM_n$ denotes the resulting change in $M_n$, then 
$\Delta_e M_n = (\alpha_{e,n}- \omega_e')^+ - (\alpha_{e,n}- \omega_e)^+$ 
where $x^+$ denotes the positive part of $x$. 
\end{lemma}
To prove this lemma, we first need to prove a well known characterization of the minimal spanning tree on a graph with distinct edge weights. Since we have assumed that the edge weight distribution is continuous, the weights of all edges and paths are automatically distinct with probability one. 
\begin{lemma}\label{charac}
An edge $e\in E_n$ belongs to the minimal spanning tree on $G_n$ if and only if $\omega_e< \alpha_{e,n}$. Moreover, if $h$ is the unique edge with weight $\alpha_{e,n}$, then the lighter of the two edges $e$ and $h$ belongs to the tree and the other one does not. 
\end{lemma}
\begin{proof}
Let $T$ denote the minimal spanning tree. First suppose that $e\in T$. Let $T_1$ and $T_2$ denote the two connected components of $T\backslash\{e\}$. There is a path in $G_n$ connecting the two endpoints of $e$, which does not contain $e$ and whose edge weights are all $\le \alpha_{e,n}$. At least one edge $r$ in this path is a bridge from $T_1$ to $T_2$. If $\omega_e> \alpha_{e,n}$, then we can delete the edge $e$ from $T$ and add the edge $r$ to get a tree that has total weight $<M_n$, which is impossible. Therefore $\omega_e < \alpha_{e,n}$. Next, suppose that $\omega_e < \alpha_{e,n}$. Let $P$ be the unique path in $T$ that connects the two endpoints of $e$. If $P$ does not contain $e$, then $P$ must contain an edge that has weight  $\ge \alpha_{e,n} > \omega_e$. Deleting this edge from $T$ and adding the edge $e$ gives a tree with weight $< M_n$, which is impossible. Hence $T$ must contain $e$.

To prove the second assertion of the lemma, first observe that if $\omega_h > \omega_e$, then $e\in T$ and $h\not \in T$ by the first part. On the other hand  if $\omega_h < \omega_e$, then $e\not \in T$ by the first part; and if $\alpha_{h,n} < \omega_h$, then there exists a path connecting the two endpoints of $e$ whose edge weights are all $< \alpha_{e,n}$, which is impossible. Therefore again by the first part, $h\in T$. 
\end{proof}

We are now ready to prove Lemma \ref{difflmm}.

\begin{proof}[Proof of Lemma \ref{difflmm}]

Let $T$ and $T'$ denote the minimal spanning trees before and after replacing $\omega_e$ by $\omega_e'$. Note that since $T$ and $T'$ are both spanning trees, we have (I): $T$ and $T'$ must necessarily have the same number of edges. 

By symmetry, it suffices to work under the assumption that  $\omega_e' < \omega_e$. Clearly, this implies that $\alpha_{h,n}' \le \alpha_{h,n}$ for all $h\in E_n$ and equality holds for $h=e$. Thus, by Lemma \ref{charac}, we make the observation (II): every edge in $T'$ other than $e$ must also belong to $T$.

Let $h$ be the unique edge that has weight $\alpha_{e,n}$.  There are three possible scenarios: (a) If $\omega_h < \omega_e' < \omega_e$, then by Lemma \ref{charac}, $e\not \in T$ and $e\not \in T'$. Therefore by the observations (I) and (II), $T=T'$. (b) If $\omega_e'< \omega_h <\omega_e$, then by Lemma \ref{charac}, $e\in T'$, $h\not \in T'$, $e\not \in T$ and $h\in T$. By (I) and (II), this means that $T'$ is obtained from $T$ by deleting $h$ and adding $e$. (c) If $\omega_e' < \omega_e< \omega_h$, then $e\in T$ and $e\in T'$, and therefore by (I) and (II), $T=T'$. In all three cases, it is easy to see that the formula for $\Delta_e M_n$ is valid. This completes the proof of Lemma \ref{difflmm}. 
\end{proof}

Lemma \ref{difflmm} gives an expression for $\Delta_e M_n$, but it does not make it obvious why this discrete difference is approximately equal to a local quantity. The secret lies in a monotonicity argument, similar in spirit to an idea from \cite{kesten}. 

\begin{lemma}\label{monotone}
For any $e\in E$, the sequence $\alpha_{e,n}$ is a non-increasing sequence, converging everywhere to a limiting random variable $\alpha_{e,\infty}$ as $n \ra\infty$. The convergence holds in $L^p$ for every $p> 0$. 
\end{lemma}
\begin{proof}
The monotonicity is clear from the definition of $\alpha_{e,n}$. Since the sequence is non-negative, the limit exists. The $L^p$ convergence holds because the random variables are bounded by a constant (since the edge weights are bounded by a constant). 
\end{proof}

Now let $c$ denote a specific edge of $E$, let's say the edge joining the origin to the point $(1,0,\ldots,0)$. For any edge $e$, let $e+V_n$ denote the set $x + [-n,n]^d \cap V_n$, where $x$ is the lexicographically smaller endpoint of $e$. In other words, $e+V_n$ is simply a translate of $V_n$ so that $0$ maps to $x$. Let $e+E_n$ be the set of edges of $e+V_n$. For each $e$, let $\beta_{e,n}$ be the smallest $\beta$ such that there is a path from one endpoint of $e$ to the other,  lying entirely in $e+V_n$ but not containing the edge $e$, such that all edges on this path have weight $\le \beta$. Clearly, $\beta_{e,n}$ has the same distribution as $\alpha_{c,n}$. The following lemma says that for a fixed edge $e$, if $n$ and $k$ and both large, and $n$ is greater than $k$, then $\alpha_{e,n}$ may be closely approximated by $\beta_{e,k}$. 
\begin{lemma}\label{monotone2}
There is a sequence $\delta_k$ tending to zero as $k\ra\infty$, such that for any $1\le k < n$ and $e\in E_{n-k}$, $
\ee|\beta_{e, k} - \alpha_{e,n}| \le\delta_k $. 
\end{lemma}
\begin{proof}
Since $e+V_{k}\subseteq V_n$, $\beta_{e,k} \ge \alpha_{e,n}$. Thus, $\ee|\beta_{e,k} - \alpha_{e,n}| = \ee(\beta_{e,k}) - \ee(\alpha_{e,n})$. But again, $V_n \subseteq e+V_{2n}$, and so $\alpha_{e,n} \ge \beta_{e, 2n}$. Thus,
\begin{align*}
\ee|\beta_{e,k} - \alpha_{e,n}| \le  \ee(\beta_{e,k}) - \ee(\beta_{e, 2n}) = \ee(\alpha_{c,k}) - \ee(\alpha_{c, 2n})\, .
\end{align*}
By Lemma \ref{monotone}, $\ee(\alpha_{c,k})$ is a Cauchy sequence. This completes the proof. 
\end{proof}
Combining Lemma \ref{monotone2} and Lemma \ref{difflmm}, we get the following corollary that gives the desired ``local approximation'' for the discrete derivatives of  $M_n$. 
\begin{corollary}\label{indep}
For any $k\ge 1$ and $e\in E$, let $\gamma_{e,k} := (\beta_{e,k} - \omega'_e)^+ - (\beta_{e,k} - \omega_e)^+$. Then for any $n > k$ and  $e\in E_{n-k}$, 
\[
\ee|\Delta_e M_n - \gamma_{e,k}|\le 2\delta_k\, ,
\]
where $\delta_k$ is a sequence tending to zero as $k \to \infty$. 
\end{corollary}

Armed with the above corollary and Corollary \ref{maincor}, we are now ready to prove Theorem \ref{mst}. 
\begin{proof}[Proof of Theorem \ref{mst}]
Throughout this proof, $C$ will denote any constant whose value depends only on the edge weight distribution and the dimension $d$. The value of $C$ may change from line to line. 

Fix an arbitrary positive integer $k$. Take any $n > k$. Take any edge $e\in E_{n-k}$, and a set of edges $A\subseteq E_n \backslash\{e\}$. Let $(\omega_h')_{h\in E_n}$ be an independent copy of $(\omega_h)_{h\in E_n}$, and just like in Theorem \ref{mainthm}, let $\omega_h^A = \omega_h$ if $h\not \in A$, and $\omega_h^A = \omega_h'$ if $h\in A$. Let $\Delta_eM_n^A$ and $\gamma_{e,k}^A$ be the values of $\Delta_eM_n$ and $\gamma_{e,k}$ in the environment $\omega^A$.

Let $h$  be any other edge in $E_{n-k}$ such that the lattice distance between $e$ and $h$ is bigger than $2k$. Let $B$ be any subset of $E_n \backslash\{h\}$. Then by Corollary \ref{indep} and the boundedness of the discrete derivatives of $M_n$ and the $\gamma$'s, we get
\[
|\cov(\Delta_eM_n \Delta_e M_n^A, \, \Delta_h M_n \Delta_h M_n^B) - \cov(\gamma_{e,k} \gamma_{e,k}^A ,\, \gamma_{h,k}\gamma_{h,k}^B) | \le C\delta_k\, .
\]
But since $(e+V_k)\cap (h + V_k)=\emptyset$, the random variables $\gamma_{e,k} \gamma_{e,k}^A$ and $\gamma_{h,k}\gamma_{h,k}^B$ are independent. In particular, their covariance is zero. Therefore,
\[
|\cov(\Delta_eM_n \Delta_e M_n^A, \, \Delta_h M_n \Delta_h M_n^B)|\le C\delta_k\, .
\]
Note that here we are only considering $e$ and $h$ in $E_{n-k}$ that are at least $2k$ apart in lattice distance. Therefore among all pairs of edges $e,h\in E_n$, we are excluding $\le Cn^{2d-1} k$ pairs from the above bound. Those that are left out, are bounded by a constant. 

All we now need is a lower bound on the variance $\sigma_n^2$. One can show that $\sigma_n^2 \ge C n^{d}$.  This requires some work, which is not necessary to present in this article. For a proof, see \cite[Section 6.5]{chatterjeesen}. Inputting this lower bound and the covariance bounds obtained in the above paragraph into Corollary \ref{maincor}, we get
\[
\sup_{t\in \rr}|\pp(f_n \le t)-\pp(Z\le t) |\le C(\delta_k + k/n)^{1/4} + Cn^{-d/4}\, . 
\]
The proof is finished by taking $n\ra\infty$ and then taking $k\ra\infty$.
\end{proof}

\section{Some open problems}\label{open}
Probability theory has come a long way in figuring out how to prove  central limit theorems. Still, there are problems where we do not know how to proceed. Many of these problems come from random combinatorial optimization. One example of a solved problem from this domain is the central limit theorem for minimal spanning trees, discussed in Section \ref{mstsec}. But there are many others that are quite intractable.  

For example, consider the Euclidean traveling salesman problem on a set of random points. Let $X_1,\ldots, X_n$ be a set of points chosen independently and uniformly at random from the unit square in $\rr^2$. Let $P$ be a path that visits all points, ending up where it started from, which minimizes the total distance traveled among all such paths. It is widely believed that the length of $P$ should obey a central limit theorem under appropriate centering and scaling, but there is no proof. 

Again, in the same setting, we may consider the problem of minimum matching. Suppose that $n$ is even, and we pair the points into $n/2$ pairs such that the sum total of the pairwise distances is minimized. It is believed that this minimum matching length should be approximately normally distributed, but we do not know how to prove that. 

One may also consider lattice versions of the above problems, where instead of points in Euclidean space we have random weights on the edges of a lattice. One can still talk about the minimum weight path that visits all points on a finite segment of the lattice, and the minimum weight matching of pairs of points. Central limit theorems should hold for both of these quantities. 

For basic results about such models, a classic reference is the monograph of Steele \cite{steele}. The reason why one may speculate that normal approximation should hold is that the solutions of these problems are supposed to be ``local'' in nature. For example, the optimal path in the traveling salesman problem is thought to be of ``locally determined''; one way to make this a little more precise is by claiming that a small perturbation at a particular location is unlikely to affect the path in some faraway neighborhood. This is the same as what we earlier called ``the approximate independence of small perturbations''. If this is proven to be indeed the case, then the generalized perturbative version of Stein's method should be an adequate tool for proving a central limit theorem. 

Mean field versions of these problems, which look at complete graphs instead of lattices or Euclidean points, have been analyzed in great depth in a remarkable set of papers by W\"astlund \cite{wastlund10, wastlund12}. In the case of minimum matching, this generalizes the famous work of Aldous \cite{aldous} on the random assignment problem. These papers, however, do not prove central limit theorems. It is an interesting question whether the insights gained from W\"astlund's works can be applied to prove normal approximation in the mean field setting by rigorously proving the independence of small perturbations. 

Another class of problems that may be attacked by high dimensional versions of Stein's method are problems of universality in physical models. There are various notions of universality; the one that is closest to standard probability theory is the following. Suppose that $Z = (Z_1,\ldots, Z_n)$ is a vector of i.i.d.\ standard normal random variables, and $X = (X_1,\ldots, X_n)$ is a vector of  i.i.d.\ random variables from some other distribution, with mean zero and variance one. Let $f:\rr^n \ra \rr$ be some given function. When is it true that $f(X)$ and $f(Z)$ have approximately the same probability distribution? In other words, when is it true that for all $g$ belonging to a large class of functions, $\ee g( f(X)) \approx \ee g( f(Z))$? The classical central limit theorem says that this is true if $f(x) = n^{-1/2}(x_1+\cdots + x_n)$. Lindeberg \cite{lindeberg22} gave an ingenious proof of the classical CLT in 1922 using the idea of replacing one $X_i$ by one $Z_i$ at a time, by an argument that I am going to describe below. 

The idea was generalized by Rotar \cite{rotar79} to encompass low degree polynomials. The polynomial version was applied, in combination with hypercontractive estimates, to solve several open questions in theoretical computer science by Mossel et al.~\cite{mossel}. 

I think I may have been the first one to realize in \cite{chatterjee05, chatterjee06} that the Lindeberg method applies to general functions (and not just sums and polynomials), with a potentially wide range of interesting applications. The basic idea is the following: Let $h = g\circ f$. For each $i$, let $U^i = (X_1,\ldots, X_i, Z_{i+1},\ldots, Z_n)$ and $V^i = (X_1,\ldots, X_{i-1}, 0, Z_{i+1},\ldots, Z_n)$. Then by Taylor expansion in the $i$th coordinate,
\begin{align*}
\ee h(U^i) - \ee h(U^{i-1})  &= \ee\biggl( h(V^i) + X_i \partial_i h(V^i) + \frac{1}{2}X_i^2 \partial_i^2 h(V^i)\biggr)\\
&\qquad -  \ee\biggl( h(V^i) + Z_i \partial_i h(V^i) + \frac{1}{2}Z_i^2 \partial_i^2 h(V^i)\biggr) + O(\|\partial_i^3 h\|_\infty)\, .
\end{align*} 
By the independence of the $X_i$'s and $Z_i$'s, and the assumptions that $\ee X_i = 0$ and $\ee X_i^2 = 1$, it follows that the two expectations on the right-hand side are equal. Therefore, summing over $i$, we get
\begin{align}\label{ehh}
\ee h(X)-\ee h(Z) = O\biggl(\sum_{i=1}^n \|\partial_i^3 h\|_\infty\biggr)\, .
\end{align}
If the right-hand side is small, then we get our desired conclusion. 

In \cite{chatterjee05, chatterjee06} I used this idea to give a new  proof of the universality of Wigner's semicircle law, and a proof of the universality of the free energy of the Sherrington-Kirkpatrick model of spin  glasses. The random matrix problems were tackled by choosing $h$ to be the Stieltjes transform of the empirical spectral distribution of the random matrix at a point $z\in \cc\backslash \rr$. By taking $z$ close to $\rr$ and overcoming some major technical difficulties that arise in the process, the method was later used with great effect in a series of papers by Tao and Vu \cite{taovu10a, taovu10b, taovu11} to prove universality of local eigenvalue statistics of several kinds of random matrices.

The connection with Stein's method comes through the following variant of the Lindeberg idea. Suppose, instead of the above, we consider a solution $w$ of the Stein equation 
\[
\Delta w(x) - x\cdot \nabla w(x) = h(x)-\ee h(Z)\, .
\]
Let $W^i := (X_1,\ldots, X_{i-1}, 0, X_{i+1},\ldots, X_n)$. Then by the independence of the $X_i$'s and the facts that $\ee X_i= 0$ and $\ee X_i^2 = 1$, Taylor expansion gives
\begin{align*}
\ee(X_i \partial_i w(X)) &= \ee (X_i \partial_i w(W^i) + X_i^2 \partial_i^2 w(W^i)) +  O(\|\partial_i^3 w\|_\infty)\\
&= \ee\partial_i^2 w(W^i) + O(\|\partial_i^3 w\|_\infty) = \ee\partial_i^2 w(X) + O(\|\partial_i^3 w\|_\infty)\, .
\end{align*}
Summing over $i$, this gives 
\[
\ee h(X)-\ee h(Z) = \ee(\Delta w (X) - X\cdot \nabla w(X)) =  O\biggl(\sum_{i=1}^n \|\partial_i^3 w\|_\infty\biggr)\, ,
\]
which is basically the same as \eqref{ehh}, except that we have third derivatives of $w$ instead of $h$. Undoubtedly, this is nothing but Stein's method in action. A version of this argument was used by Carmona and Hu \cite{carmonahu} to prove the universality of the free energy in the Sherrington-Kirkpatrick model, at around the same time that I proved it in \cite{chatterjee05}. Sophisticated forms of this idea have been used by Erd\H{o}s, Yau and coauthors in their remarkable series of papers \cite{erdos10a, erdos10b, erdos11, erdos12} proving universality of random matrix eigenvalue distributions, running parallel to the papers of Tao and Vu, who used the Lindeberg approach. This  demonstrates the potential for high dimensional versions of Stein's method to prove universality. There are still many problems where we do not know how to establish universal behavior (for example, last- and first-passage percolation, various polymer models, gradient Gibbs measures, etc.). It would be interesting to see Stein's method being used to attack such problems.

\vskip.2in
\noindent{\bf Acknowledgments.} I thank Susan Holmes  and Persi Diaconis for many useful comments on the first draft of this manuscript.


\end{document}